\documentclass[reqno,12pt]{amsart}
\usepackage{amsmath,amssymb,epsfig,color,url}

\numberwithin{equation}{section}

\newtheorem{theorem}{Theorem}[section]
\newtheorem{conjecture}[theorem]{Conjecture}
\newtheorem{lemma}[theorem]{Lemma}
\newtheorem{proposition}[theorem]{Proposition}
\newtheorem{definition}[theorem]{Definition}
\newtheorem{corollary}[theorem]{Corollary}

\theoremstyle{remark}
\newtheorem{example}[theorem]{Example}

\newcommand{\vanish}[1]{}

\begin{document}

\title[Efron's coins]{Efron's coins and the Linial arrangement}

\author[G\'abor Hetyei]{G\'abor Hetyei}

\address{Department of Mathematics and Statistics,
  UNC-Charlotte, Charlotte NC 28223-0001.
WWW: \tt http://www.math.uncc.edu/\~{}ghetyei/.}


\dedicatory{To $\mbox{(Richard P.\ S.\ )}^2$}


\date{\today}

\begin{abstract}
We characterize the tournaments that are dominance graphs of sets of
(unfair) coins in which each coin displays its larger side with greater
probability. The class of these tournaments coincides with the class of
tournaments whose vertices can be numbered in a way that makes them
semiacyclic, as defined by Postnikov and Stanley. We provide an example
of a tournament on nine vertices that can not be made semiacyclic, yet it may be
represented as a dominance graph of coins, if we also allow coins that
display their smaller side with greater probability. We conclude with an
example of a tournament with $81$ vertices that is not the dominance
graph of any system of coins. 
\end{abstract}

\maketitle

\section*{Introduction}

A fascinating paradox in probability theory is due to B.\ Efron, who
devised a four element set of nontransitive dice~\cite{Gardner}: the
first has four faces labeled $4$ and two faces labeled $0$, the second has
all faces labeled $3$, the third has four faces labeled $2$ and two
faces labeled $6$, the fourth has three faces labeled $1$ and three
faces labeled $5$. On this list, die number $i$ defeats die number $i+1$
in the cyclic order in the following sense: when we roll the pair of dice
simultaneously, die number $i$ is more likely to display the larger number
than die number $i+1$. The paradox arose the interest of Warren
Buffet, it inspired several other similar constructions and papers 
in game theory and probability: sample references include
~\cite{Alon-etc,Bednay-Bozoki,Gervacio-Maehara,Honsberger,Savage}.

This paper intends to investigate a hitherto unexplored aspect of Efron's
original example: on each of its dice, only at most two numbers
appear. They could be replaced with unfair coins, which display one of
two numbers with a given probability. This restriction seems to be
strong enough that we should be able to describe exactly which
tournaments can be realized as {\em dominance graphs} of collections of
unfair coins, where the direction of the arrows indicates which coin of
a given pair is more likely to display the larger number. 

 Our paper gives a complete characterization in the case
when we restrict ourselves to the use of {\em winner} coins: these are
coins that are more likely to display their larger number. The answer,
stated in Theorem~\ref{thm:winners}, is that a tournament has such a
representation exactly when its vertices may be numbered in a way that
it becomes a {\em semiacylic tournament}. These tournaments were
introduced by Postnikov and Stanley~\cite{Postnikov-Stanley}, the number
of semiacyclic tournaments on $n$ numbered vertices is the same as the
number of regions of the $(n-1)$-dimensional Linial arrangement. Our
result allows to construct an example of a tournament on $9$-vertices
that can not be represented using winner coins only. On the other hand,
we will see that this example is representable if we also allow the use
of {\em loser} coins, that is, coins that are more likely to display
their smaller number. However, as we will see in Theorem~\ref{thm:prod},
any tournament that is not representable with a set of winner coins
only, gives rise via a direct product operation to a tournament that is
not representable by any set of coins. In particular, we obtain an
example of a tournament on $81$ vertices that can not be represented by
any set of coins as a dominance graph.
Our results motivate several open questions, listed in the concluding
Section~\ref{sec:conc}.

\section{Preliminaries}

A hyperplane arrangement is a finite collection of codimension one hyperplanes
in a finite dimensional vectorspace, together with the induced partition
of the space into regions. The number of these regions may be expressed
in terms of the M\"obius function in the intersection poset of the
hyperplanes, using Zaslavsky's formula~\cite{Zaslavsky}. 

The {\em Linial arrangement} ${\mathcal L}_{n-1}$ is the hyperplane arrangement
\begin{equation}
\label{eq:Linial}  
x_i-x_j=1, \quad 1\leq i<j\leq n
\end{equation}
in the $(n-1)$-dimensional vector space $V_{n-1}=\{(x_1,\ldots,x_n)\in
{\mathbb R}^n \::\: x_1+\cdots + x_n=0\}$. We will use the 
combinatorial interpretation of the regions of ${\mathcal L}_{n-1}$ 
in terms of {\em semiacyclic tournaments}, due
to Postnikov and Stanley~\cite{Postnikov-Stanley}.  A {\em tournament}
on the vertex set $\{1,\ldots,n\}$ is a directed graph with no loops nor
multiple edges, such that for each $2$-element subset $\{i,j\}$ of
$\{1,\ldots,n\}$, exactly one of the directed edges $i\rightarrow j$ and
$j\rightarrow i$ belongs to the graph. We may think of a tournament as
the visual representation of the outcomes of all games in a championship,
such that each team plays against each other team exactly once, and
there is no ``draw''.
\begin{definition} 
A directed edge $i\rightarrow j$ is called an {\em ascent} if $i<j$ and
it is a {\em descent} if $i>j$. For any directed cycle
$C=(c_1,\ldots,c_m)$ we denote the number of directed edges that are
ascents in the cycle by $\operatorname{asc}(C)$, and the number of
directed edges that are descents by $\operatorname{desc}(C)$. A cycle is {\em ascending}
if it satisfies $\operatorname{asc}(C)\geq \operatorname{desc}(C)$. A
tournament on $\{1,\ldots,n\}$ is {\em semiacyclic} if it contains no
ascending cycle.  
\end{definition}
To each region $R$ in ${\mathcal L}_{n-1}$ we may associate a tournament 
on $\{1,\ldots,n\}$ as follows: for each $i<j$ we set $i\rightarrow j$
if $x_i>x_j+1$ and we set $j\rightarrow i$ if $x_i<x_j+1$. Postnikov and
Stanley, and independently Shmulik Ravid, gave the following
characterization of the tournaments arising this way~\cite[Proposition
  8.5]{Postnikov-Stanley} 
\begin{proposition}
\label{prop:semiacyclic}
A tournament $T$ on $\{1,\ldots,n\}$ corresponds to a region $R$ in
${\mathcal L}_{n-1}$ if and only if $T$ is semiacyclic. Hence the number
$r({\mathcal L}_{n-1})$ of regions of ${\mathcal L}_{n-1}$ is the number
of semiacyclic tournaments on $\{1,\ldots,n\}$. 
\end{proposition}

\section{The coin model and its elementary properties}

In this paper we will study $n$ element sets of (fair and unfair)
coins. Each coin is described by a triplet of real parameters $(a_i,b_i,
x_i)$ where $a_i\leq b_i$ and $x_i>0$ hold (here $i=1,2,\ldots,n$). The
$i$th coin has the number $a_i$ on one side and $b_i$ on the
other.  After flipping it, it shows the number $a_i$ with probability
$1/(1+x_i)$, equivalently it shows the number $b_i$ with probability
$x_i/(1+x_i)$. Note that, as $x_i$ ranges over the set of all positive
real numbers, the probability $1/(1+x_i)$ ranges over all numbers in the
open interval $(0,1)$. We call the triplet $(a_i,b_i,x_i)$ the {\em
  type} of the coin. We say that coin $i$ {\em dominates} coin $j$ if,
after tossing both at the same time, the probability that coin $i$
displays a strictly larger number than coin $j$ is greater than
the probability that coin $j$ displays a strictly larger number. In
other words, when we flip both coins, the one displaying the larger
number ``wins'', the other one ``loses'', and we consider both coins
displaying the same number a ``draw''. The coin that is more likely to
win, dominates the other.  

We represent the domination relation as the {\em dominance
  graph}, whose vertices are the coins and there is a directed edge
$i\rightarrow j$ exactly when coin $i$ dominates coin $j$. We will be
interested in the question, which tournaments may be represented as the
dominance graph of a set of $n$ coins. 

Up to this point we made one, inessential simplification: we assume
that $x_i$ can not be zero or infinity, that is, no coin can
land on the same side with probability $1$. If we have such a coin, we may
replace it with a coin of type $(a,a,1)$, that is, a fair coin that has
the same number written on both sides. Since we are interested only in
dominance graphs as tournaments we will also require that {\em for each pair of
coins one dominates the other}. After fixing the parameters $a_i$
and $b_i$ for each coin, this restriction will exclude the points of
$\binom{n}{2}$ hypersurfaces of codimension one from the set of possible
values of $(x_1,x_2,\ldots,x_n)$, each hypersurface being defined by an
equation involving a pair of variables $\{x_i, x_j\}$. The equations
defining these surfaces may be obtained by replacing the inequality
symbols with equal signs in the inequalities stated in
Table~\ref{tab:dominance} below. In particular,  we assume that {\em
  different coins have different types}. In the rest of this section we
will describe in terms of the types when the $i$th coin dominates the
$j$th coin. To reduce the number of cases to be considered we first show
that we may assume that no coin has the same number written on both
sides. This is a direct consequence of the next lemma.

\begin{lemma}
\label{lem:2values}
\end{lemma}  
Suppose there is a coin of type $(a_i,b_i,x_i)$, satisfying
$a_i=b_i$. Replace this coin with a coin of type $(a_i',b_i, 1)$  where
$a_i'$ is any real number that is less than $a_i$ but larger than any
element in the intersection of the set $\{a_1, b_1, a_2, b_2,
\ldots, a_n, b_n\}$ with the open interval $(-\infty, a_i)$. Then
modified system of coins has the same dominance graph.
\begin{proof}
We only need to verify that another coin $j$, of type $(a_j,b_j,x_j)$
dominates a coin of type $(a_i,a_i,x_i)$ if and only if it dominates
a coin of type $(a_i',a_i,1)$. This is certainly the case when neither
of $a_j$ and $b_j$ is equal to $a_i$, as these numbers compare to $a_i$
the same way as to $a_i'$. We are left to consider the case when
exactly one of $a_j$ and $b_j$ equals $a_i$ (both can not equal because
then there is no directed edge between $i$ and $j$ in the dominance
graph).

\noindent{\bf Case 1:} $a_j=a_i$ and so $a_i'<a_i=a_j<b_j$ hold. 
In the original system, as well as in the modified one, only coin $j$
can win against coin $i$ and it does so with a positive probability. We
have $j\rightarrow i$ in both dominance graphs. 

\noindent{\bf Case 2:} $b_j=a_i$ and so $a_j<a_i'<a_i=b_j$. In the
original system only coin $j$ can lose (when it displays $a_j$) so we
have $i\rightarrow j$ in the dominance graph. In the modified system,
$i$ can also lose sometimes, exactly when it displays $a_i'$ (with
probability $1/2$) and coin $j$ displays $a_j$. The probability of this
event is $1/(1+x_j)/2$. A draw can also occur, exactly when both coins
display $a_i=b_j$, the probability of this event is $x_j/(1+x_j)/2$.
By subtracting these probabilities from $1$  we obtain that the
probability that $j$ loses is exactly $1/2$. This is more than the
probability of $i$ losing, as $1/(1+x_j)/2<1/2$. Thus $i\rightarrow j$
still holds in the dominance graph of the modified system. 
\end{proof}
From now on we will assume that the type $(a_i,b_i, x_i)$ of each coin
satisfies $a_i<b_i$. If the type also satisfies $x_i>1$ then we call the
coin a {\em winner}, if it satisfies $x_i<1$, we call it a {\em
  loser}. Note that a loser coin can dominate a winner coin under the
appropriate circumstances, these terms refer to the fact whether the
coin is more likely to display its larger or smaller value. Note that a
coin satisfies $x_i=1$ exactly when it is a {\em fair coin}. 

In several results we will list our coins in {\em increasing
  lexicographic order of their types}.
\begin{definition}
We say that the type $(a_i,b_i,x_i)$ is lexicographically smaller than
the type $(a_j,b_j,x_j)$, if one of the following holds:
\begin{enumerate}
\item $a_i<a_j$;
\item $a_i=a_j$ and $b_i<b_j$; 
\item $a_i=a_j$, $b_i=b_j$ and $x_i<x_j$.
  \end{enumerate}
\end{definition}  

\begin{theorem}
  Assume we list the coins in increasing lexicographic order by their
  types, comparing the coordinates left to right. Then, for $i<j$, we
  have $i\rightarrow j$ if and only if exactly one of the following
  conditions is satisfied: 
\begin{enumerate}
\item $a_i=a_j<b_i<b_j$ and $1/x_j >1/x_i +1$;
\item $a_i<a_j<b_j<b_i$ and $x_i>1$;
\item $a_i<a_j<b_i=b_j$ and $x_i>x_j+1$;
\item $a_i<a_j<b_i<b_j$ and $(1/x_i +1)(x_j+1) < 2$.
\end{enumerate}    
\end{theorem}
\begin{proof}
Assuming that $(a_i,b_i)\leq (a_j,b_j)$ holds in the lexicographic
order, the six cases corresponding to the six lines of
Table~\ref{tab:dominance} below are a complete
and pairwise mutually exclusive list of possibilities.
\begin{table}[h]
\begin{tabular}{|c|c|}
\hline    
Relation between $(a_i,b_i)$ and $(a_j,b_j)$ & $i\rightarrow j$ exactly
when\\
\hline
\hline
$(a_i,b_i)=(a_j,b_j)$ & $x_i>x_j$\\
\hline
$a_i=a_j<b_i<b_j$ & $1/x_j >1/x_i +1$\\
\hline
$a_i<a_j<b_j<b_i$ & $x_i>1$\\
\hline
$a_i<a_j<b_i=b_j$ & $x_i>x_j+1$\\
\hline
$a_i<a_j<b_i<b_j$ & $(1/x_i +1)(x_j+1) < 2$\\
\hline
$a_i<b_i\leq a_j <b_j$ & never\\
\hline
\end{tabular}  
\caption{Characterization of the dominance relation in terms of the types}
\label{tab:dominance}
\end{table}
The statement on the first line of Table~\ref{tab:dominance} is obvious,
and $x_i>x_j$ can not happen as $(a_i,b_i,x_i)$ comes
before $(a_i,b_i,x_j)$ in the lexicographic order exactly when $x_i<
x_j$ holds. The description on
the third line is also clear, coin $i$ wins exactly when it displays
$b_i$, otherwise it loses.  The last line is also obvious: coin $j$ can
not lose in that case and it wins with a positive probability. The
remaining lines require just a little more attention.   

To prove the statement on the second line, observe first that
coin $i$ wins exactly when it displays $b_i$ and coin $j$ displays
$a_j$, and it loses exactly when coin $j$ displays $b_j$ (regardless the
outcome of tossing coin $i$). Thus $i\rightarrow j$ exactly when
$$
\frac{x_i}{x_i+1}\cdot \frac{1}{x_j+1}>\frac{x_j}{x_j+1}
$$
which is equivalent to $1/x_i+1<1/x_j$. The statement on the fourth line
is completely analogous, only easier. 

We are left to prove the statement on the fifth line. In this case, coin
$i$ wins exactly when it displays $b_i$ and coin $j$ displays $a_j$, and
it loses in all other cases (there is never a draw). Therefore coin $i$
dominates coin $j$ if and only if the probability of $i$ winning is more
than $1/2$, that is, we have
$$
\frac{x_i}{x_i+1}\cdot \frac{1}{x_j+1}>\frac{1}{2}.
$$
This is obviously equivalent to the statement on the fifth line. 

\end{proof}

\section{Winner coins represent semiacyclic tournaments}

In this section we give a complete description of all tournaments that
may be realized as the domination graph of a system of coins that has
only winner and fair coins. In proving our main result, the
following lemma plays an important role.

\begin{lemma}
\label{lem:bdec}
Assume that the $i$th coin, of type $(a_i,b_i,x_i)$ dominates the $j$th coin, of
type $(a_j,b_j,x_j)$ and that the $j$th coin is not a loser coin. Then we
must have $b_i\geq b_j$.
\end{lemma}
\begin{proof}
Assume by way of contradiction that $b_j>b_i$ holds. The $j$th coin wins
whenever it displays $b_j$, and this happens with probability
$x_j/(1+x_j)\geq 1/2$. We obtain that $i\rightarrow j$ can not hold, a
contradiction.
\end{proof}

\begin{theorem}
\label{thm:winners}
Assume a set of $n$ winner and fair coins is listed in increasing lexicographic
order of their types. If the domination graph is a tournament, it must be
semiacyclic. Conversely  every semiacyclic tournament is the domination
graph of a set of winner coins. 
\end{theorem}  
\begin{proof}
Assume we are given an $n$-element set of winner and fair coins whose
domination graph is a tournament and that the coins are listed in
increasing lexicographic order of their types. Consider a cycle
$(i_1,i_2,\ldots,i_k)$ in the domination graph. Repeated use of
Lemma~\ref{lem:bdec} yields $b_{i_1}\geq b_{i_2}\geq \cdots \geq
b_{i_k}\geq b_{i_1}$, implying $b_{i_1}=\cdots=b_{i_k}$. As a
consequence, using the fourth line in Table~\ref{tab:dominance}, we have
$x_{i_{j+1}}<x_{i_j}-1$ whenever $i_j<i_{j+1}$ and
$x_{i_{j+1}}<x_{i_{j}}+1$ whenever $i_{j+1}<i_j$, for
$j=1,\ldots,k-1$. Similarly, we have $x_{i_{1}}<x_{i_k}-1$ whenever
$i_k<i_{1}$ and $x_{i_{1}}<x_{i_{k}}+1$ whenever $i_{1}<i_k$. In other
words, the last coordinate of the type decreases by more than $1$ at
each ascent, and it increases by less than $1$ at each
descent. Therefore, a cycle can not be an ascending cycle, and the
tournament must be semiacyclic.

Conversely, consider a semiacyclic tournament $T$ on $\{1,\ldots,n\}$.
By Proposition~\ref{prop:semiacyclic} this tournament corresponds to a
region of the Linial arrangement ${\mathcal L}_{n-1}$ in the
following way. For each $i<j$ we have $i\rightarrow j$
if $x_i>x_j+1$ and we have $j\rightarrow i$ if $x_i<x_j+1$, where
$(x_1,\ldots,x_n)$ is an arbitrary point in the region. Note that some of the
coordinates must be negative or zero here, as we have
$x_1+\cdots+x_n=0$. Introducing $r:=\max\{ c-x_1, \ldots, c-x_n\}$ for
some $c>1$, 
and setting $x_i'= x_i+r$, we obtain a vector $(x_1',\ldots, x_n')$ whose
coordinates are all greater than $1$ and satisfy $i\rightarrow j$
if $x_i'>x_j'+1$ and $j\rightarrow i$ if $x'_i<x'_j+1$, for each
$i<j$. Consider now the set of $n$ coins where the $i$th coin has type
$(i,n+1,x_i')$. By the fourth line of Table~\ref{tab:dominance}, the
dominance graph of this set of coins is precisely $T$.  
\end{proof}
The proof of Theorem~\ref{thm:winners} also has the following consequence.
\begin{corollary}
\label{cor:winners}
If a set of coins does not contain loser coins and its dominance
graph is a tournament, then this tournament is also the dominance graph
of a set of winner coins that all have the same number on one of their
sides. 
\end{corollary}
Analogous results may also be stated for loser and fair coins. In
analogy to Lemma~\ref{lem:bdec} we can make the following observation.
\begin{lemma}
\label{lem:adec}
Assume that the $i$th coin, of type $(a_i,b_i,x_i)$ dominates the $j$th coin, of
type $(a_j,b_j,x_j)$ and that the $i$th coin is not a winner coin. Then we
must have $a_i\geq a_j$.
\end{lemma}
The proof is completely analogous and omitted. Lemma~\ref{lem:adec} may
  be used to prove the following result.
\begin{theorem}
\label{thm:losers}
Assume a set of $n$ loser and fair coins is listed in increasing lexicographic
order of their types. If the domination graph is a tournament, it must be
semiacyclic. Conversely  every semiacyclic tournament is the domination
graph of a set of loser coins. 
\end{theorem}  
\begin{proof}
  The proof is analogous to that of Theorem~\ref{thm:winners}. Consider
  first the dominance graph of a set of $n$ loser and fair coins and
  assume it is a tournament. We may use
 Lemma~\ref{lem:adec} to show that, in this dominance graph, all
 vertices contained in a cycle must have the same first coordinate. The
 role played by the fourth line of Table~\ref{tab:dominance} is taken
 over by the second line, which may be used to show that the reciprocal
 of the third coordinate of the type increases by more than $1$ after
 each ascent, and it decreases by less than $1$ after each
 descent. Again we obtain that the dominance graph can not contain an
 ascending cycle.

 Conversely, given a semiacyclic tournament $T$ on the
 vertex set $\{1,\ldots,n\}$, consider a point $(x_1,\ldots,x_n)$
 satisfying $x_1+\cdots+x_n=0$, such that  for each $i<j$ we have
 $i\rightarrow j$ if $x_i>x_j+1$ and we have $j\rightarrow i$ if
 $x_i<x_j+1$. Let us set $r:=\max\{ x_1+c, \ldots,
 x_n+c\}$ for some $c>1$, and let us define $x_i'$ by $x_i'=1/(r-x_i)$ for
 $i=1,\ldots,n$. Consider the set of $n$ coins where the $i$th coin has
 type $(0,i,x_i')$ for $i=1,\ldots,n$. Since $1/x_i'=r-x_i>1$ holds for
 all $i$, all coins are loser coins. Furthermore $1/x_j'>1/x_i'+1$ is
 equivalent to $r-x_j>r-x_i+1$, that is, $x_i>x_j+1$ for each
 $i<j$. Hence the dominance graph of this set of coins is $T$.  
\end{proof}  
In analogy to Corollary~\ref{cor:winners}, 
the proof of Theorem~\ref{thm:losers} also has the following consequence.
\begin{corollary}
\label{cor:losers}
If a set of coins does not contain winner coins and its dominance
graph is a tournament, then this tournament is also the dominance graph
of a set of loser coins that all have the same number on one of their
sides. 
\end{corollary}
We conclude this section by an example of a tournament $T$ whose vertices
can not be labeled in an order that would make $T$ semiacyclic. As a
consequence, $T$ can not be the dominance graph of a set of coins that
does not contain winner, as well as loser coins. Our example will use a
direct product construction, which we will reuse later, thus we make a
separate definition.
\begin{definition}
  \label{def:ttimes}
Given two tournaments $T_1$ and $T_2$ on the vertex sets $V_1$ and $V_2$
respectively, we define their {\em direct product $T_1\times T_2$} as
follows. Its vertex set is $V_1\times V_2$ and we set $(u_1,u_2)\rightarrow
(v_1,v_2)$ if either $u_1\rightarrow v_1$ belongs to $T_1$ or we have
$u_1=v_1$ and $u_2\rightarrow v_2$ belongs to $T_2$.   
\end{definition}  
This direct product operation is not commutative, but it is associative
in the following sense: given $k$ tournaments $T_1, \ldots, T_k$ on
their respective vertex sets $V_1,\ldots, V_k$, the vertex set of
$T_1\times \cdots \times T_k$ (where parentheses may be inserted in any order)
is identifiable with the set of all $k$-tuples $(v_1,\ldots, v_k)$,
where, for each $i$, $v_i$ belongs to $V_i$. The edge
$(u_1,\ldots,u_k)\rightarrow (v_1,\ldots,v_k)$ belongs to
$T_1\times \cdots\times T_k$ exactly when $u_i\rightarrow v_i$ belongs
to $T_i$ for the least $i$ such that $u_i\neq v_i$.

\begin{proposition}
\label{prop:3x3}
Let $C_3$ denote the $3$-cycle. Then for any numbering of the vertex set
of $C_3\times C_3$, the resulting labeled tournament is not semiacyclic. 
\end{proposition}
\begin{proof}
  Let us identify the vertices of $C_3$ with $0$, $1$ and $2$ in the order
  that $0\rightarrow 1$, $1\rightarrow 2$, $2\rightarrow 0$ belong to
  $C_3$. The vertex set of $C_3\times C_3$ is then the set of ternary
  strings of length $2$. Assume, by way of contradiction, that there is
  a labeling of these ternary strings that makes $C_3\times C_3$ a
  semiacyclic tournament.

  For each $u\in \{0,1,2\}$ the cyclic permutation 
$(u0,u1,u2)$ is an automorphism of the tournament $(C_3\times
  C_3)$. Hence, without loss of generality, we may assume that in the
  set $\{u0,u1,u2\}$ the vertex $u0$ has the least label. As a
  consequence, the edge $u0\rightarrow u1$ is an ascent for each
  $u$. The cycle $(00, 01, 10, 11, 20, 21)$ has $6$ edges, and contains
  at least $3$ ascents, thus it is an ascending cycle, in contradiction
  of having a semiacyclic tournament.
\end{proof}

\section{General sets of coins}
\label{sec:gen}

We begin with showing that the tournament $C_3\times C_3$ introduced in 
Proposition~\ref{prop:3x3} can be represented with a set of coins, if we allow
both winner and loser coins.
\begin{example}
\label{ex:33rep}
We represent the vertices of $C_3\times C_3$ with the set
of coins given in Table~\ref{tab:33rep}.
\begin{table}[h]
\begin{tabular}{|l||c|c|c|}
\hline
vertex & $00$&$01$&$02$\\
coin $(a,b,x)$ &  $(3,6,1+\delta)$&   $(2,6,3/2)$& $(1,6,2+2\delta)$\\
\hline
vertex &   $20$& $21$& $22$\\
coin $(a,b,x)$&  $(5,12,1/(s+1+2\varepsilon))$&    $(5,11,1/(s+1/2)))$&
    $(5,10,1/(s+\varepsilon))$\\   
\hline
vertex & $10$& $11$ & $12$\\
coin $(a,b,x)$&   $(4,9,1/(r+1+2\varepsilon))$& $(4,8,1/(r+1/2))$&
  $(4,7,1/(r+\varepsilon))$\\
\hline
\end{tabular}
\caption{A coin-representation of the vertices of $C_3\times C_3$}
\label{tab:33rep}
\end{table}

Here $r$, $s$, $\varepsilon$ and $\delta$ are positive real numbers, whose
values we determine below. Initially we only require $r>1$ and $s>1$ which makes
all coins loser except for the ones representing $(0,0)$, $(0,1)$ and
$(0,2)$. 

To assure $00\rightarrow 01\rightarrow 02\rightarrow 00$ we need
$$2+2\delta < 3/2+1,\quad 3/2 <1+\delta +1,\quad \mbox{and}\quad
2+2\delta> 1+\delta +1.$$
All three are satisfied if and only if we have
\begin{equation}
  \label{eq:delta}
0<\delta<1/4.
\end{equation}  
To assure $20\rightarrow 21\rightarrow 22\rightarrow 20$ we need
$$r+1/2< r+\varepsilon+1,\quad r+1+2\varepsilon < r+1/2+1,\quad \mbox{and}\quad
r+1+2\varepsilon> r+\varepsilon +1.$$
All three are satisfied if and only if we have
\begin{equation}
  \label{eq:eps}
0<\varepsilon<1/4.
\end{equation}  
Note that condition \eqref{eq:eps} is independent of the value of $r$,
and we may replace $r$ with $s$ in the above calculations. Therefore to
assure $10\rightarrow 11\rightarrow 12\rightarrow 10$ we also need to
make sure \eqref{eq:eps}, and this is a sufficient condition.  

Note next, that any coin $(4,b_1,x_1)$ representing a vertex of the
form $2v_1$ is dominated by any coin $(5,b_2,x_2)$ representing a vertex of
the form $1v_2$. Indeed, with probability $1/(1+x_1)>1/2$ the first
coin displays $4$ and loses.

Next we make sure that any coin $(4,b_1,x_1)$ representing a vertex of
the form $2v_1$ dominates  any coin $(a_2,6,x_2)$ representing a vertex
of the form $0v_2$. We have $a_2<4<6<b_1$ and, using line 5 of
Table~\ref{tab:dominance} we get that
$$
(1+1/x_2)(1+x_1)>2
$$
needs to be satisfied. Substituting the least value o $x_1$ and the
largest value of $x_2$ we get
$$
\left(1+\frac{1}{2+2\delta}\right)\left(1+\frac{1}{r+1+2\varepsilon}\right)>2
$$
This inequality is equivalent to
\begin{equation}
  \label{eq:r}
r<\frac{2}{1+\delta}-2\varepsilon. 
\end{equation}
Finally, we want to make sure that any coin $(5,b_1,x_1)$ representing a
vertex of the form $2v_1$ is dominated by any coin $(a_2,6,x_2)$
representing a vertex of the form $0v_2$.
We have $a_2<5<6<b_1$ and, using line 5 of
Table~\ref{tab:dominance} we get that
$$
(1+1/x_2)(1+x_1)<2
$$
needs to be satisfied. Substituting the largest value o $x_1$ and the
least value of $x_2$ we get
$$
\left(1+\frac{1}{1+\delta}\right)\left(1+\frac{1}{s+\varepsilon}\right)<2.
$$
This inequality is equivalent to
\begin{equation}
  \label{eq:s}
s>\frac{2+\delta}{\delta}-\varepsilon. 
\end{equation}
The conditions \eqref{eq:delta}, \eqref{eq:eps}, \eqref{eq:r} and
\eqref{eq:s} may be simultaneously satisfied, by setting $\delta=0.1$,
$\varepsilon=0.1$, $r=1.6$ and $s=22$, for example.   
\end{example}  

Example~\ref{ex:33rep} proves that among the dominance graphs of systems
of coins there are some that can not be labeled to become semiacyclic
tournaments. On the other hand, Theorems~\ref{thm:winners} and \ref{thm:losers}
imply an important necessary condition for a tournament to be the
dominance graph of a system of coins.
\begin{corollary}
\label{cor:winlose}
If a tournament $T$ may be represented as the dominance graph of a
system of coins, then its vertex set $V$ may be written as a union
$V=V_1\cup V_2$, such that the full subgraphs induced by $V_1$  and
$V_2$, respectively, may be labeled to become semiacyclic tournaments. 
\end{corollary}  
Indeed, given a system of coins whose dominance graph is $T$, we may
choose $V_1$ to be the set of winner and fair coins and $V_2$ to be the
set of loser and fair coins (we do not have to include the fair coins on
both sets) and then apply Theorems~\ref{thm:winners} and
\ref{thm:losers}.   

We conclude this section with an example of a tournament that can not be
the dominance graph of any system of coins. Our example is $C_3\times
C_3\times C_3\times C_3$. The fact this tournament is not the dominance
graph of any system of coins, is a direct consequence of
Proposition~\ref{prop:3x3} and the next theorem. 

\begin{theorem}
\label{thm:prod}
Suppose the tournaments $T_1$ and $T_2$ have the property that they are
not semiacyclic for any ordering of their vertex sets. Then the
tournament $T_1\times T_2$ can not be the dominance graph of any system
of coins.
\end{theorem}  
\begin{proof}
Let us denote the vertex set of $T_1$ and $T_2$, respectively, by $V_1$
and $V_2$, respectively. Assume, by way of contradiction, that
$T_1\times T_2$ is the dominance graph of a system of coins. As noted in
Corollary~\ref{cor:winlose}, we may then write $V_1\times V_2$ as a
union $V_1\times V_2=W_1\cup W_2$ such that the restriction of
$T_1\times T_2$ to either of $W_1$ or $W_2$ can be labeled to become a
semiacyclic tournament. For a fixed $v_1\in V_1$, the restriction of
$T_1\times T_2$ to the set $\{(v_1,v_2)\::\: v_2\in V_2\}$ is isomorphic
to $T_2$. Indeed, the first coordinate is the same for all vertices in
the set and identifying each $(v_1,v_2)$ with $v_2$ yields an
isomorphism. Since $T_2$ can not be ordered to be semiacyclic, we obtain
that $\{(v_1,v_2)\::\: v_2\in V_2\}$ can not be entirely contained in
$W_2$. As a consequence, for each $v_1\in V_1$ we can pick an $f(v_1)\in
V_2$ such that $(v_1,f(v_1))$ belongs to $W_1$. Consider now the
restriction of $T_1\times T_2$ to the set $\{(v_1,f(v_1))\::\: v_1\in
V_1\}$. This restriction is isomorphic to $T_1$, an isomorphism is given
by $(v_1,f(v_1))\mapsto v_1$. Since $T_1$ can not be labeled to become a
semiacyclic tournament, neither can the restriction of $T_1\times T_2$
to the set $W_1$ that properly contains $\{(v_1,f(v_1))\::\: v_1\in
V_1\}$. We obtained a contradiction.      
\end{proof}  

\section{Concluding remarks}
\label{sec:conc}

The problem of representing tournaments with sets of dice is closely
related to representing tournaments by {\em voting preference
  patterns}. In the latter setup the $n$ vertices of the tournament
correspond to $n$ candidates. There are $m$ voters, and each has a
linearly ordered preference list of all candidates. Candidate $i$
defeats candidate $j$ if the majority of voters prefers $i$ over $j$.
Equivalently we may instruct voter $i$ to assign the ``score''
$(i-1)\cdot n+ n+1-k$ to the the $k$th candidate on their preference
list.
We may then associate an $m$-sided fair die to each candidate in such
a way that each face corresponds to a voter and is marked by the score
the voter assigned to the candidate. The dominance graph of this set of
dice is identical with the tournament of the voting preference
patterns. It has been shown by McGarvey~\cite{McGarvey} that every
tournament on $n$ vertices can be represented as a preference pattern of
$n$ candidates and $n(n-1)$ voters. A lower bound of $0.55 n/\log(n)$
on the minimally necessary number of voters to be able to represent all
tournaments on $n$ vertices was given by
Stearns~\cite{Stearns}. Erd\H os and Moser~\cite{Erdos-Moser} have
shown that any tournament on $n$ vertices may be also realized as a
preference pattern of $O(n/\log(n))$ voters. As a consequence, any
tournament may be represented as the dominance graph of a set of dice
with $O(n/\log(n))$ faces. However, not all sets of dice need to arise in
connection with voting preference patterns, and Bednay and
Boz\'oki~\cite{Bednay-Bozoki} showed that every tournament on $n$
vertices can be realized with a set of dice such that each die has
$\lfloor 6n/5\rfloor $ faces. Our paper shows that $2$ faces do not
suffice, even if we allow the ``two-sided dice'' (better known as coins)
to be unfair, and even if we allow ``ties'' between two coins and only
require the dominating coin to display the larger number with a greater
probability. On the other hand, our research indicates that describing
the classes of tournaments that may be represented by sets of dice with
a fixed number of sides could be an interesting question. 

Theorem~\ref{thm:winners} motivates the question, how to describe those
tournaments that can not be made semiacyclic, no matter how we order
their vertices. Postnikov and Stanley have given a list of five
cycles~\cite[Theorem 8.6]{Postnikov-Stanley}, one of which must appear
if the tournament (with numbered vertices) is not semiacyclic. Our
question is different here, as we are allowed to choose our own
numbering on the vertices.
\begin{conjecture}
There is a finite list of tournaments that are minimally not semiacyclic
in the sense, that removing any of their vertices allows the numbering
of the remaining vertices in a semiacyclic way.
\end{conjecture}
We wish to make the analogous conjecture for dominance graphs of systems 
of coins. 
\begin{conjecture}
\label{conj:gen}
There is a finite list of tournaments that are minimally not dominance 
graphs of systems of coins in the sense, that after removing any of their
vertices we obtain the dominance graph of a system of coins.
\end{conjecture}
Proving Conjecture~\ref{conj:gen} may be helped by looking at
generalizations of hyperplane arrangements to hypersurface arrangements,
that allow adding equations of hypersurfaces of the form 
$1/x_j = 1/x_i +1$ and of the form $(1/x_i +1)(x_j+1) = 2$, see
Table~\ref{tab:dominance}.  Such a study would only allow finding an
analogue of semiacyclic tournaments (with numbered vertices), the
question would still be open which subgraphs would obstruct numbering
the vertices in a way that allows them to be represented in the desired
way. 

The number of regions in the Linial arrangement is listed as sequence
A007889 in the OEIS~\cite{OEIS}. These numbers count alternating trees
and binary search trees as well. Several ways are known to find these
numbers, see ~\cite{Athanasiadis-extLin,Postnikov,Postnikov-Stanley}
neither of which seems to be related to counting semiacyclic tournaments
directly. In view of the role played by these tournaments, it may be
interesting to find a way to count them directly.  

\section*{Acknowledgments}
This work was partially supported by a grant from the Simons Foundation
(\#245153 to G\'abor Hetyei). The author thanks Clifford Smyth for
helpful comments. The author is grateful to an anonymous referee for the
very careful reading of this manuscript and for suggesting substantial
improvements.


\begin{thebibliography}{99}

\bibitem{Alon-etc}
N.\ Alon, G.\ Brightwell, H.\ A.\ Kierstead,  A.\ V.\ Kostochka, and
P.\ Winkler,  
Dominating sets in k-majority tournaments, 
{\it J.\ Combin.\ Theory Ser.\ B}  {\bf 96} (2006), 374--387. 

\bibitem{Athanasiadis-extLin}
C.\ A.\ Athanasiadis, 
Extended Linial hyperplane arrangements for root systems and a
conjecture of Postnikov and Stanley, 
{\it J.\ Algebraic Combin.\ } {\bf 10} (1999), 207--225. 

\bibitem{Bednay-Bozoki}
  D.\ Bednay and S.\ Boz\'oki,
  Constructions for Nontransitive Dice Sets,
  preprint, 8 pages\\ \url{http://eprints.sztaki.hu/7623/1/Bednay_15_2507773_ny.pdf}

\bibitem{Erdos-Moser}
P.\  Erd\H os, and L.\ Moser, 
On the representation of directed graphs as unions of
orderings, 
{\it Magyar Tud. Akad.\ Mat.\ Kutat\'o Int.\ K\"ozl.\ } {\bf 9} (1964),
125--132. 

  
\bibitem{Gardner}
M.\ Gardner,
Mathematical Games: The Paradox of the Nontransitive Dice and the
Elusive Principle of Indifference,
{\it Sci.\ Amer.\ } {\bf 223} (1970), 110--114.


\bibitem{Gervacio-Maehara} 
S.\ V. Gervacio, Severino and H. Maehara, 
Partial order on a family of k-subsets of a linearly ordered set, 
{\it Discrete Math. } {\bf 306} (2006), 413--419. 

\bibitem{Honsberger}
R.\   Honsberger,
Some Surprises in Probability, Ch. 5 in Mathematical Plums
(Ed.\ R.\ Honsberger), Washington, DC: Math. Assoc. Amer., pp. 94--97,
1979.  

\bibitem{McGarvey}
David C.\ McGarvey, 
A theorem on the construction of voting paradoxes, 
{\it Econometrica} {\bf 21} (1953), 608--610. 

\bibitem{Moon-Moser}
J.\ W.\ Moon, and L.\ Moser, 
Generating oriented graphs by means of team comparisons, 
{\it Pacific J.\ Math.\ }  {\bf 21} (1967), 531--535.   


\bibitem{OEIS}
OEIS Foundation Inc.\ (2011), ``The On-Line Encyclopedia of Integer Sequences,''
published electronically at \url{http://oeis.org}.


\bibitem{Postnikov}
A.\ Postnikov, 
Intransitive trees,
{\it J.\ Combin.\ Theory Ser.\ A} {\bf 79} (1997), 360--366. 

\bibitem{Postnikov-Stanley}
A.\ Postnikov and R.\ P.\ Stanley, 
Deformations of Coxeter hyperplane arrangements,
{\it J.\ Combin.\ Theory Ser.\ A} {\bf 91} (2000), 544--97. 

\bibitem{Savage}
R.\ P.\ Savage, 
The Paradox of Nontransitive Dice,
{\it  The American Mathematical Monthly} 
{\bf 101} (1994), 429--436. 

\bibitem{Stearns}
R.\ Stearns, 
The voting problem, 
{\it Amer.\ Math.\ Monthly} {\bf 66} (1959), 761--763.   

\bibitem{Zaslavsky}
T.\ Zaslavsky,
Facing up to arrangements: face-count formulas for partitions of space
by hyperplanes,  
{\it Mem.\ Amer.\ Math.\ Soc.\ } {\bf 1} (1975), issue 1, no. 154, vii+102 pp. 

\end{thebibliography}
\end{document}